\documentclass{lmcs} 
\pdfoutput=1

\usepackage{lastpage}
\lmcsdoi{18}{3}{24}
\lmcsheading{}{\pageref{LastPage}}{}{}%
{Mar.~30,~2021}{Aug.~19,~2022}{}


\usepackage{hyperref}
\usepackage{mathtools}
\usepackage{booktabs}
\usepackage[utf8]{inputenc}
\def\str{\texttt{Str}}

\def\coi{{^{co}\mathbf{I}}}
\def\R{\mathbf{R}}
\def\sd{\mathbf{Sd}}
\def\total{\mathbf{T}}
\def\DC{\mathbf{DC}}

\begin{document}

\title[Limits of signed digit streams]{Limits of real numbers in the binary signed digit representation
}

\author[F.~Wiesnet]{Franziskus Wiesnet\lmcsorcid{0000-0003-3870-6984}}[a]

\author[N.~K{\"o}pp]{Nils K{\"o}pp\lmcsorcid{0000-0002-8280-4744}}[b]	

\address{University of Trento, Via Sommarive 14, 38123 Povo and
Ludwig-Maximilians Universit{\"a}t, Theresienstr. 39, 80333 M{\"u}nchen and TU Wien,
Favoritenstraße 9-11,
1040 Wien}	
\email{franziskus.wiesnet@tuwien.ac.at }
\address{Ludwig-Maximilians Universit{\"a}t, Theresienstr. 39, 80333 M{\"u}nchen}	
\email{koepp@math.lmu.de}  

\begin{abstract}
We extract verified algorithms for exact real number computation from constructive proofs.
To this end we use a coinductive representation of reals as streams of binary signed digits.
The main objective of this paper is the formalisation of a constructive proof that real numbers are closed with respect to limits.
All the proofs of the main theorem and the first application are  implemented  in  the Minlog  proof  system and the extracted terms are further translated into Haskell.
We compare two approaches.
The first approach is a direct proof.
In the second approach we make use of the representation of reals by a Cauchy-sequence of rationals.
Utilizing translations between the two represenation and using the completeness of the Cauchy-reals, the proof is very short.

In both cases  we  use  Minlog's  program  extraction  mechanism to  automatically  extract  a  formally  verified  program  that  transforms  a converging sequence of reals, i.e.~a sequence of streams of binary signed digits together with a modulus of convergence, into the binary signed digit representation of its limit. The correctness of the extracted terms follows directly from the soundness theorem of program extraction.  

As a first application we use the extracted algorithms together with Heron's method
to construct an algorithm that computes square roots with respect to the binary signed digit representation.
In a second application we use the convergence theorem to show that the signed digit representation of real numbers is closed under multiplication.
\end{abstract}

\maketitle

\section{Introduction and motivation}\label{S:one}
\subsection{Real numbers}
Real numbers can be represented in several ways. One of the best-known representations is as Cauchy sequences of rational numbers together with a Cauchy modulus. Namely a \emph{Cauchy real} is a pair $((a_n)_n,M)$ consisting of a sequence $(a_n)_n$ of real numbers and a modulus $M:\mathbb{Z}^+\to \mathbb{N}$ such that $\forall_{p}\forall_{n,m> M(p)}|a_n-a_m|\leq 2^{-p}$, i.e.~$(a_n)_n$ is a Cauchy sequence with modulus $M$.

However, in this paper the representation of real numbers as Cauchy reals will be just a tool. The main theorems of this paper are concerned with the signed digit representation of real numbers.
\subsection{Binary representation vs. signed digit representation}
The binary representation of a real number $x$ in $[-1,1]$ is given by
\[
x = s\sum_{i=1}^{\infty}a_i2^{-i},
\]
where $s\in \{-1,1\}$ and $a_i\in \{0,1\}$ for every $i$. Here and further on, by equality $=$ between two reals we mean an equivalence relation that is compatible with the usual operations and relations on the reals. In reality the specific of the real equality depends on the representation of real numbers.
The binary representation of some concrete real number corresponds to a sequence of nested intervals.
Reading the digits one after the other the interval is halved in each step.
Hence from the binary code we can approximate a real number to arbitrary precision.
\begin{figure}[htbp] 
  \centering
     \includegraphics[width=\textwidth]{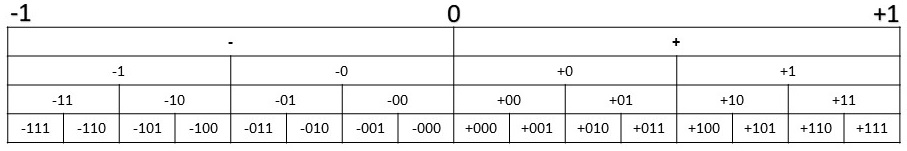}
  \caption{Visualization of the binary code}
  \label{fig:binary}
\end{figure}
Now consider the other direction, i.e.~given a real number, compute the binary representation.
This is not always possible, since the $\leq$-relation is not decidable. Further it is not possible to e.g.~compute the binary representation of $\frac{x+y}{2}$ given representation of $x$ and $y$. Here ``compute'' means that there is an algorithm which takes as input the binary streams of $x$ and $y$ and generates the binary stream representing $\frac{x+y}2$. In particular, the algorithm can only use finitely many binary digits of $x$ and $y$ in order to generate finitely many binary digits of $\frac{x+y}{2}$. For example, it is not possible to compute even the first digit (i.e.~+ or -) of the average of $+\vec{0}???\cdots$ and $-\vec{0}???\cdots$, where $\vec{0}$ is a list with entries $0$ of arbitrary length and $?$ stands for an unknown digit. This is not possible due to the ``gaps'' in the binary representation. They are illustrated in Figure \ref{fig:binary} at $0$, $\frac 12$, $- \frac 12$, $\frac{1}{4}$ and so on. From the first digit of a representation of a real $x$, we can decide $0\leq x$ or $x\leq 0$, which in general can not be done if reasoning constructively about reals.
The signed digit code fills these gaps. For a real number $x\in [-1,1]$ it is defined by
\[
x=\sum_{i=1}^{\infty}d_i2^{-i},
\]
where $d_i\in \{\overline{1},0,1\}$ for every $i$.
\begin{figure}[htbp] 
  \centering
     \includegraphics[width=\textwidth]{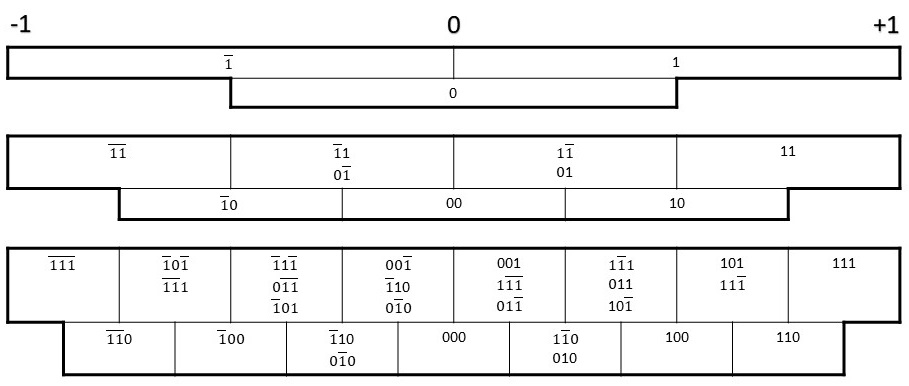}
  \caption{Visualization of the signed digit code}
  \label{fig:SD}
\end{figure}
As the illustration in Figure \ref{fig:SD} makes clear, to compute the first signed digit of a real number $x\in [-1,1]$ we have to decide which of the cases $x\leq 0$, $-\frac 12\leq 0 \leq \frac 12$ or $0\leq x$ holds. Now this is possible by application of the comparability theorem
\[
	\forall_{x,y,z}\left(x<y \to z\leq y \vee x\leq z\right).
\]
Figure \ref{fig:SD} also shows that the SD code of a real number (except $-1$ and $1$) is not unique, whereas the binary code is ``almost'' everywhere unique.

A stream of signed digits is an infinite list $d_1d_2d_3\dots$ of elements in 
$$\textbf{Sd} := \{\overline{1},0,1\}.$$
 We will not use the signed digit streams directly, rather we use a coinductively defined predicate ${^{co}\textbf{I}}$, which is given in the next section. For a real number $x$ a realiser of $x\in{^{co}\textbf{I}}$ is a signed digit stream representing $x$. The desired algorithms are given by the extracted terms of the proofs.
The soundness theorem of program extraction \cite{schwichtenberg2018computational, SchwichtenbergWainer12, Masterarbeit} gives correctness of these algorithms.
\subsection{Historical background}
One of the first papers where  signed digits are used to represent real numbers, was published by Edwin Wiedmer in 1980  \cite{Wiedmer80}.
 The idea to use coinductive algorithms to describe the operators on the reals goes back to Alberto Ciaffaglione and Pietro Di Gianantonio \cite{CiaffaglineGianantonio} and was revised by Ulrich Beger and Tie Hou \cite{berger2008coinduction,bergercoind2011}. The idea to use coinductively defined predicates together with the soundness theorem in this context is due to Ulrich Berger and Monika Seisenberger \cite{BergerSeisenberger12}. The notation and definitions in this paper are taken from \cite{MiyamotoSchwichtenberg15} written by Kenji Miyamoto and Helmut Schwichtenberg. For the implementation of the translations between signed-digit and Cauchy-representation in Minlog see \cite{KoeppMaster}.
\subsection{Implementation in Minlog}
For computing the extraced terms and verifying the correctness of the proofs, the proof assistant Minlog \cite{MinlogPage} is used. An introduction to Minlog can be found in \cite{wiesnet2018introduction} or \texttt{doc/tutor.pdf} in the Minlog directory. The implementation of the proofs can be found in the file \texttt{examples/analysis/sdlim.scm} in the directory of Minlog.
After each proof we state its computational content not in the notation of Minlog but in the notation of Haskell, since the runtime of the programs in Haskell is shorter, and the terms can be defined in a more readable way. 
So after each proof we give the extracted term of the proof which was translated to Haskell using the command \texttt{terms-to-haskell-program}.
 \subsection{Procedure of this paper}
In the next section we define a coinductively defined predicate $\coi$ on reals. Its computational interpretation is that a real number $x\in\coi$ has a signed digit representation. Then we prove some basic properties about it. In the second part of this chapter, we introduce the predicate $\R$  on reals. The computational interpretation $\R x$ is the existence of a sequence $as$ of rationals and a modulus $M$ such that $as$ converges to $x$ with modulus $M$. We conclude the second section by showing that $[-1,1]\cap \R$ and $\coi$ are equivalent.

The third section contains two proofs of the main theorem. We show that the limit of a converging sequence in $\coi$ is again in $\coi$. The first proof is a direct proof. Computationally it operates on the signed digit stream of real numbers only. In the second proof we use the equivalence between $\coi$ and $\R$ and that $\R$ is closed under limit, which is known as the completeness of Cauchy-reals. In both cases the computational content of the proof is a function which takes a stream of signed digit streams and a modulus and returns a new signed digit stream. 

The last section contains two applications of the convergence theorem. To show the square root of a real number in $\coi$ is again in $\coi$ we use Heron's method and the convergence theorem. Lastly we consider the multiplication of two reals numbers in $\coi$.
By representing one factor as a limit of reals we obtain a multiplication
program for signed digit streams as a simple iteration of the average function.

\section{Formalisation}\label{Formalisation}
\subsection{The theory of computational functionals (TCF)}
We use the formal theory TCF to formalize statements like ``$x$ is represented by some signed digit stream". In this section we give a short overview of TCF. For a complete and formal introduction we refer to \cite{SchwichtenbergWainer12, Masterarbeit}.

In TCF all terms are typed. Types in TCF are either type variables, function types or algebras. Algebras can be seen as fixpoints of their constructors. For examples, the type $\mathbb{N}$ of natural numbers is the algebra with the constructors $0: \mathbb{N}$ and $S:\mathbb{N}\to\mathbb{N}$. In short notation we express this as $\mathbb{N}:=\mu_{\xi}(\xi,\xi\to \xi)$, where $\mu$ is interpreted as least-fixed-point operator. 
Since each variable comes with a type, we will use the following naming conventions to supress type declarations.
\begin{nota}
The following table shows which variables have which type.
\begin{alignat*}{3}
	&m,n: \mathbb{N}\qquad\qquad
	&a,b: \mathbb{Q}\qquad\qquad
	&M,N: \mathbb{Z}^+\to \mathbb{N} \\
	&d,e,k: \mathbb{Z}\qquad\qquad
	&x,y: \mathbb{R}\qquad\qquad
	&as,bs: \mathbb{N}\to\mathbb{Q}  \\
	&p,q: \mathbb{Z}^+ \qquad\qquad
	&v,u: \mathbb{S} \qquad\qquad
	&xs,ys: \mathbb{N}\to \mathbb{R}
\end{alignat*}
\noindent
If other variables are used, their type is either not relevant or we declare it individually. 
\end{nota}
Here, $\mathbb{Z}^+$ is defined as the positive (binary) numbers, $\mathbb{Z}$ as the integers,  $\mathbb{Q}$ as the rational numbers and $\mathbb{R}$ as real numbers. How these algebras are defined in detail however is not important for our purpose.
In particular, in Minlog the type of real numbers $\mathbb{R}$ is explicitly defined as the type $(\mathbb{N}\to \mathbb{Q})\times (\mathbb{Z}^+\to \mathbb{N})$. Since in the following proofs the concrete representation of real numbers is not important, we view $\mathbb{R}$ as an abstract datatype and assume that we have abstract axiomatized reals with the usual operations including addition, multiplication, less-than and an equivalence-relation such that the other operations are compatible with it. We will refer to explicit representations by using predicates e.g.~$x\in\mathbf{R}$ and the computational content of a proof of this statement is a witness that $x$ has an $\mathbf{R}$-representation. 

In TCF predicates are defined (co-)inductively. Each inductively defined predicate comes with introduction axioms, also called clauses, and an elimination axiom. A coinductively defined predicate will be given by a closure axiom and a coinduction axiom. An example of an inductivly defined predicate is the totality predicate $\mathbf{T}$ which we discuss below or the predicate $\mathbf{R}$ defined in Section  \ref{sec:CauchyReals}. In Section \ref{Sec:CoindDef} we introduce a coinductively defined predicate regarding the dinged digit representation. This coinductively defined predicate is the main reason why TCF is the most suitable as underlying theory for our purpose.
For an unary predicate $A$, we write  $t\in A$  for $At$ and $\forall_{t\in A}B$, $\exists_{t\in A}B$ are short for $\forall_{t}(At\to B)$ and $\exists_t(At\wedge B)$, respectively. Examples for these abbreviations that we later use include $x\in\coi$, $x\in\R$, $d\in\mathbf{SD}$.

Note that in TCF the existence quantifier as well as the conjunction and the distinction are formally inductively defined predicates. For examples, $A\wedge B$ is defined by the clause $A\rightarrow B \rightarrow A\wedge B$, in short notation $A\wedge B := \mu_X(A\to B \to X)$. As this notation suggest, $A\wedge B$ is the least predicate $X$ which fulfills $A\to B \to X$. Furthermore, we have $A\vee B := \mu_X(A\to X, B\to X)$ and $\exists_t A := \mu_X(\forall_t(A\rightarrow X))$.

Another important property of TCF is that a term with a certain algebra as type does not have to consist of finitely many constructors of this type. For example, a natural number $n: \mathbb{N}$ does not have to be in the form $S\dots S 0$. This means that terms in TCF are partial in general. E.g.~we can also consider an infinite natural number which behaves like $SSS\dots$. However, we can not longer prove statements of the form $\forall_tA(t)$ by induction on $t$ as we do not know how $t$ is constructed and hence we can ad hoc not use something like induction on natural numbers. In order to use induction after all, we will use the totality predicate $\total$ of TCF. Informally speaking, $t\in\total$ for some term $t:\tau$  means that $t$ is a finite constructor expression if $\tau$ is an algebra, or $t$ maps total object to total objects, if $\tau$ is a function type. E.g.~for a sequence of natural numbers $ns:\mathbb{N}\to\mathbb{N}$ we have $ ns\in\total:=\forall_{n\in\total}(ns\ n)\in\total$ where $n\in\total$ is the inductive predicate given by the clauses $0\in\total$ and $n\in\total\to (n+1)\in\total$. The elimination axiom of $n\in\total$ is induction over natural numbers. Formally, the totality predicate is defined by recursion over the type. In particular, for each type we have an individual totality predicate. However, we do not mention this explicitly in the notation. For example, we just write $n \in \mathbf{T}$ instead of $n\in \textbf{T}_{\mathbb{N}}$, as the type is clear from the context. We furthermore assume that predicates like $\leq$ on natural numbers or (positive) integers are defined for total objects only. In particular, if we write something like $\forall_{n\geq M(p)}A$, we mean $\forall_{n}(n \in \mathbf{T}\wedge n\geq M(p)\to A)$.
\subsection{Program extraction from proofs}
In this section we give an overview on the process of program extraction from proofs in TCF. For formal definitions we refer to \cite{Masterarbeit,SchwichtenbergWainer12,schwichtenberg2018computational,KoeppMaster}.
In this short section we do not give formal definition as they are quite complex and we will use the proof assistant Minlog in any case to carry out the program extraction.

The computational content arises from the (co-)inductively defined predicates. When defining an (co-)inductively defined predicate or a predicate variable, it must also be determined whether it is computationally relevant (cr) or non computational (nc).
For example, the totality predicate $\mathbf{T}$ is defined as computationally relevant. The same goes for the predicates $\coi$, $\R$ and $\mathbf{SD}$, which we introduce later. The equality and inequality on real numbers are non computational.
A formula is computational relevant (cr), if its last conclusion is $A\vec{t}$ where $A$ is a computational relevant predicate. Note that the universal and existence quantifier by themselves will not carry computational content, in particular the type of the formulas $A$ and $\forall_x A$ are the same. But we will use the abbreviations $\forall_{t\in A}$ and $\exists_{t\in A}$, where $t\in A$ is computationally relevant (as long as $A$ is). In the last section we said that we use abstract axiomatized real numbers. Here, we require that the axioms are non computational. This is the case for a usual axiomatisation as the axioms are about equations and inequalities.

In a first step, from a cr formula $A$ the type \textit{type} $\tau(A)$ and \textit{realizer predicate} $A^\mathbf{r}$ are defined. Formally, this is done by recursion on the structure of the formula. The realiser predicate is a predicate which takes a term of the type of the formula and states that a term is a realizer of the formula, i.e.~it adheres to the computational requirements stated in the formula. 

In a second step, the \textit{extracted term} $\operatorname{et}(M)$ of the formalized proof $M$ of $A$ is computed. The extracted term is a $\lambda$-term with the type of the formula and is defined by recursion over proofs. It represents the extracted algorithm from the formal proof. In our case we will state this term after each proof which was formalized in Minlog, translated to the notation of Haskell.

In the last step of program extraction we generate a proof that the extracted term is indeed a realizer of the realizer predicate, i.e.~$A^\mathbf{r}\operatorname{et}(M)$. This is the so-called soundness proof. Note that this proof can be generated automatically in Minlog.

In a nutshell, the result of formal program extraction is an algorithm in the form of a $\lambda$-term and the proof of its correctness. However, as it is hardly possible to describe a formal proof on paper, we use the proof assistant Minlog. In Minlog the last three steps above can be done automatically, so the laborious part is to find the right formulation of the theorem and the implementation of the constructive proof in Minlog. For the right formulation, we use the predicate $\coi$ which is given in the next section.
\subsection{Coinductive definition of the signed digit representation}
\label{Sec:CoindDef}
\begin{defi}[sd-code representation]
We define ${^{co}\textbf{I}}$ as the greatest fixed point of the operator
\[
\Phi(X):= \left\{x\ \middle|\ \exists_{d\in\mathbf{SD},x'}\left(Xx'\wedge |x'|\leq 1\wedge x=\frac{d+x'}2\right)\right\}.
\]
\end{defi}
A realiser of ${^{co}\textbf{I}}x$ has the type  $$\tau({^{co}\textbf{I}})=\mu_{\tau(X)}(\tau(\Phi(X))\rightarrow \tau(X))=\mu_{\xi}(\textbf{Sd}\rightarrow \xi \rightarrow \xi ).$$ Here we have identified $\tau(\textbf{Sd})=\mu_{\xi}(\xi,\xi,\xi)$ with $\textbf{Sd}$ itself. We define $\texttt{Str}:=\tau({^{co}\textbf{I}})$ and by $\mathrm{C}$ we denote the only constructor of $\texttt{Str}$. In Haskell notation \textbf{Sd} and \texttt{Str} are given by 
\begin{verbatim}
  data Sd = SdR  | SdM  | SdL
  data Str = C Sd Str
\end{verbatim}
In this notation we see that an element $\mathrm{C} d v$ is a $\sd$-stream with first digit $d$ and tail $v$. Sometimes we abbreviate $\mathrm{C} d v$ by just writing $dv$. We will also use this notation for reals: If we write something like $dx$ for a real number $x$ and a signed digit $d$, we  mean $\frac{d+x}{2}$.

\noindent The definition of ${^{co}\textbf{I}}$ as greatest fixpoint of $\Phi$ can be expressed by the two axioms 
\begin{align*}
{^{co}\textbf{I}}^-:\ &{^{co}\textbf{I}}\subseteq \Phi({^{co}\textbf{I}})\\
{^{co}\textbf{I}}^+:\ &X\subseteq \Phi({^{co}\textbf{I}}\cup X)\rightarrow X\subseteq {^{co}\textbf{I}},
\end{align*}
where $X$ is an unary predicate variable on real numbers. It is called \emph{competitor predicate}.
The first axiom ${^{co}\textbf{I}}^-$ says that ${^{co}\textbf{I}}$ is a fixpoint of $\Phi$. Expressed in elementary formulas it is given by
\[
\forall_{x\in {^{co}\textbf{I}}}\exists_{d\in\mathbf{Sd},x'}\left(x'\in{^{co}\textbf{I}}\wedge |x'|\leq 1\wedge x=\frac{d+x'}2\right).
\]
The type of this axiom is $\tau({^{co}\textbf{I}}^-)=\texttt{Str}\rightarrow \textbf{Sd}\times \texttt{Str}$ and a realiser is the destructor $\mathcal{D}$ given by the computation rule
\[
\mathcal{D}(\mathrm{C}dv) := \langle d,v\rangle.
\]
The destructor takes a stream and returns a pair consisting of its first digit and its tail. Using the projectors $\pi_0$ and $\pi_1$ one gets the first digit and the tail, respectively.
E.g.~consider a cr formula of the form $x\in\coi\to A$ of type $\texttt{Str}\to\tau(A)$. Now assume that in its proof, $\coi^-$ is used with $x\in\coi$ at some point. On the computational level this corresponds to reading the head of the input stream and storing its tail.

\noindent
The second axiom ${^{co}\textbf{I}}^+$ expresses that ${^{co}\textbf{I}}$ is the greatest fixpoint in a strong sense. It is explicitly given by:
\[
\label{eq:coind}
\tag{$\star$}
\forall_{x}\left(Xx\to
\forall_{x}\left(Xx\to \exists_{d\in\mathbf{Sd},x'}\left(x'\in\left({^{co}\textbf{I}}\cup X\right)\wedge|x'|\leq 1\wedge x=\dfrac{d+x'}{2}\right)\right)\rightarrow x\in{^{co}\textbf{I}}\right)
\]
The type depends on the type of the predicate variable $X$, namely
\[
\tau({^{co}\textbf{I}}^+)= \tau(X) \rightarrow \left(\tau(X)\rightarrow \textbf{Sd}\times(\texttt{Str}+\tau(X))\right)\rightarrow \texttt{Str}.
\]
A realiser of ${^{co}\textbf{I}}^+$ is the corecursion operator $^{co}\mathcal{R}$ which is given by the computation rule
\[
^{co}\mathcal{R}t f := \begin{cases}
	\mathrm{C} (\pi_0(ft))v&\text{if\ }\pi_1(f t)=\text{in}_0(v)\\
	\mathrm{C}(\pi_0(ft))^{co}\mathcal{R}t' f&\text{if\ }\pi_1(f t)=\text{in}_1(t').
\end{cases}
\]
Here $\text{in}_0$ and $\text{in}_1$ are the two constructors of the type sum $\texttt{Str}+\tau(X)$.
If $\pi_1(ft)$ has the form $\text{in}_0v$, the corecursion stops and we have $\mathrm{C} (\pi_0(ft))v$ as signed digit representation. If it has the form $\text{in}_1t'$, the corecursion continues with the new argument $t'$. In both cases we have obtained at least the first digit $\pi_0(ft)$ of the stream. By iterating the corecursion we can generate each digit one by one. In Haskell we have
\begin{verbatim}
  strDestr :: (Str -> (Sd, Str))
  strDestr (C s u) = (s , u)

  strCoRec :: (alpha -> ((alpha -> (Sd, (Either Str alpha))) -> Str))
  strCoRec g h = (C (fst (h g)) 
                    (case (snd (h g)) of
	                      { Left u0 -> u0 ;
		                        Right g1 -> (strCoRec g1 h) })).
\end{verbatim}
Moreover we sometimes use the following functions.
\begin{verbatim}
  hd :: (Str -> Sd)
  hd u = fst(strDestr u)
  
  tl :: (Str -> Str)
  tl u = snd(strDestr u)
  
  sdtoint :: Sd -> Integer
  sdtoint SdR = 1
  sdtoint SdM = 0
  sdtoint SdL = -1
  
  id :: Pos -> Nat
  id p = p
\end{verbatim}
Here \texttt{fst} and \texttt{snd} are the pair-projections.
\subsection{Basic lemmas}
We prove two basic lemmas, which will often occur in the proofs following:
\begin{lem}[\texttt{CoICompat}]\label{Lem:CoICompat}
The predicate $\coi$ is compatible with real equality, i.e. 
\[
\forall_{x\in\coi}\forall_{y}\left( x=y \rightarrow y\in{^{co}\textbf{I}}\;\right)
\]
\end{lem}
\begin{proof}
We apply \eqref{eq:coind} to the predicate $Px:=\exists_{y\in\coi}(x=y)$:
\[
\forall_x\left(Px\rightarrow
\forall_x\left(Px \rightarrow\exists_{d\in\mathbf{Sd},x'}\left(x'\in\left({^{co}\textbf{I}}\cup P\right)\wedge|x'|\leq 1\wedge x=\dfrac{d+x'}{2}\right)\right)
\rightarrow x\in{^{co}\textbf{I}}\right)
\]
It is sufficient to prove the second premise. So assume $x$ and $y\in{^{co}\textbf{I}}$ with $x=y$ are given. Using $y\in{^{co}\textbf{I}}$ with $\coi^-$ we get $e\in \textbf{Sd}$ and $y'\in {^{co}\textbf{I}}$ with $|y'|\leq 1$ and $y=\frac{e+y'}2$. Hence $d:=e$ and $x':=y'$ have the desired properties.
\end{proof}
In the following proofs, this theorem is used tacitly. In Minlog it has the name \texttt{CoICompat}.
The extracted term of this theorem is given by
\begin{verbatim}
  cCoICompat :: (Str -> Str)
  cCoICompat u0 = strCoRec u0 (\ su1 -> (case su1 of
                                          {s2 u2 -> (s2,Left u2)}))
\end{verbatim}
Assume \texttt{f} is the costep-function above, then
for \texttt{C s u} some stream we have \texttt{f(C s u) = (s,Left u)}. So if we unfold the \texttt{strCoRec} we get
\begin{verbatim}
  cCoICompat (C s u) = C s u,
\end{verbatim}
i.e.~the computational content of this lemma is actually the identity. Hence to increase readability, we will leave it out in the following.
\begin{lem}[\texttt{CoIClosureInv}]\label{Lem:CoIClosure}
\[
\forall_{x\in\coi,d\in\mathbf{Sd}}\frac{d+x}2\in \coi
\]
\end{lem}
\begin{proof}
We use \eqref{eq:coind} with the predicate
\[
Px := \exists_{d\in\mathbf{Sd},x'\in{^{co}\textbf{I}}}x=\frac{d+x'}{2},
\]
Again, in order to prove the goal formula, it is sufficient to prove the second premise. Therefore our goal is
\[
\forall_x\left(\exists_{d\in\mathbf{Sd},x'\in\coi}\left( x=\frac{d+x'}2 \right) \rightarrow \exists_{d\in\mathbf{Sd},x'}\left(x'\in\left(\coi\cup P\right)\wedge|x'|\leq 1\wedge x=\dfrac{d+x'}{2}\right)\right).
\]
But this follows immediately from $\coi\subseteq {^{co}\textbf{I}}\cup P$ and $x\in\coi\to |x'|\leq 1$.
\end{proof}
In Minlog this lemma has the name \texttt{CoIClosureInv} and its extracted term is given by
\begin{verbatim}
  cCoIClosureInv :: (Sd -> (Str -> Str))
  cCoIClosureInv s0 u1 = strCoRec (s0 , u1) 
                                  (\ su2 -> (
                                    (case su2 of
                                    { (,) s u -> s }) , 
                                    (Left (case su2 of
                                      { (,) s0 u0 -> u0 })))),
\end{verbatim}
which is an elaborate way to write the constructor $\mathrm{C}$, namely if \texttt{f} is the costep-function above then \texttt{f (s0,u1) = (s0,Left u1)} and by unfolding \texttt{strCoRec}
\begin{verbatim}
  cCoIClosureInv s0 u1 = C s0 u1.
\end{verbatim} 
\subsection{Cauchy reals and signed digit streams}
\label{sec:CauchyReals}
We now formalize the relation between reals represented by Cauchy-sequences of rationals and signed digit streams.
\begin{defi}[Cauchy representation]
	\label{def:Cauchy}
	We denote $as:\mathbb{N}\to\mathbb{Q}$ and $M:\mathbb{Z}^+\to\mathbb{N}$ and define 
	\[
		\mathbf{Mon} (M) := M\in\total \wedge\forall_{p\leq q}\left(Mp\leq Mq\right).	
	\]
	For a real $x$ we write $x\in\R$, if there exist $M\in \mathbf{Mon}$ and $as\in\total$ with
	\[
		\forall_{p\in\textbf{T}}\forall_{n\geq M(p)}\left(|x-(as\ n)|<2^{-(p+1)}\right),
	\]
	i.e.~there is a sequence of rationals converging to $x$.
\end{defi}
\noindent In Haskell this representation is given by the datatype
\begin{verbatim}
  data Rea = RealConstr (Nat -> Rational) (Pos -> Nat),
\end{verbatim}
with the pair-projections
\begin{verbatim}
  realSeq :: (Rea -> (Nat -> Rational))
  realSeq (RealConstr as m) = as
	
  realMod :: (Rea -> (Pos -> Nat))
  realMod (RealConstr as m) = m.
\end{verbatim}
 
In the following we will prove that for some $x\in\coi$ and $n\in\total$ there exists  a rational approximation to $x$ with precision $\frac{1}{2^n}$. To get a sequence of rationals representing $x$ we need dependent choice. Moreover later we will use countable choice.
\begin{defi} [Choice Principles]
	\label{def:Choice}
	We denote $f:\mathbb{N}\to\alpha$, then the axiom of dependent choice $\DC$ is given by
	\[
	\exists_{\alpha}P(0,\alpha) \to 
	\forall_{n\in\total,\alpha}\left(P(n,\alpha) \to \exists_{\alpha} P(n+1, \alpha)\right) \to 
	\exists_{f} \forall_{n\in\total} P(n ,f n).
	\]
	It has the type
	\[
		\tau(P)\to\left(\mathbb{N}\to \tau(P)\to \tau(P)\right)\to \mathbb{N}\to \tau(P)
	\]
	and the realizer is given by the recursion operator for $\mathbb{N}$, i.e.
	\begin{verbatim}
	  natRec :: Nat -> a -> (Nat -> a -> a) -> a
	  natRec 0 g h = g
	  natRec n g h | n > 0 = h (n - 1) (natRec (n - 1) g h).
	\end{verbatim}
	The axiom of countable choice \textbf{CC} is given by
	\[
		\forall_{n\in\mathbf{T}}\exists_{\alpha\in\mathbf{T}} P(n,\alpha) \to \exists_{f\in\mathbf{T}}\forall_{n\in\mathbf{T}} P(n,fn)
	\]
	Its type is $(\mathbb{N}\to \alpha\times\tau(P))\to(\mathbb{N}\to \alpha)\times(\mathbb{N}\to\tau(P))$ and the realizer is basically given by the identity, namely
	\[
		\lambda_F \langle \lambda_n(F\,n)_0,\lambda_n(F\,n)_1\rangle.
	\]
\end{defi}
\begin{thm}[\texttt{StrToCs}]
	\label{thm:StrToCs}
	Any real $x$ represented by a signed digit code can be represented by a Cauchy-sequence, i.e.
	\[
		\forall_{x\in\coi}( x\in\R\wedge |x|\leq 1).		
	\]
\end{thm}
	\begin{proof}
		Assume $x\in\coi$, then $|x|\leq 1$ holds by $\coi^-$. Now let
			\[
				P(n,a):=a\in\total\wedge 	\exists_{y\in\coi} \left(x=2^{-(n+1)}y+a\right)
			\] 
		We want to apply $\DC$, hence we first prove $\exists_a P(0,a)$.
		By $\coi^-$ there is $y\in\coi,d\in\mathbf{Sd}$ with $x=\frac{y+d}{2}$, so let $a:= \frac{d}{2}$.
		For the second premise of $\DC$ assume $n\in\total, a\in\total$ and $P(n,a)$ i.e.~there exists $y\in\coi$ with $x=\frac{1}{2^{n+1}}y+a$. We need to prove
			\[
				\exists_b\left( b\in\total \wedge \exists_{z\in\coi}\left( x=2^{-(n+2)}z+b\right)\right)		
			\]
		Since $y\in\coi$ we get $d\in\mathbf{Sd},y'\in\coi$ with $y=\frac{y'+d}{2}$.
		Now we choose $b:= a+\frac{d}{2^{n+2}}$ and $z=y'$, then
		\[
			x=\frac{1}{2^{n+1}}y+a=\frac{1}{2^{n+2}}y'+b
		\]
	Hence by $\DC$ there exists $as\in\total$ with
	\[
		\forall_{n\in\total} \exists_{y\in\coi} \left(x=2^{-(n+1)}y+(as\ n)\right).
	\]
	Hence with $Mp:=p$ and $|y|\leq 1$ we get $x\in\R$.
	\end{proof}	
The extracted term of the proof is given by
\begin{verbatim}
  cStrToCsInit :: (Str -> (Rational, Str))
  cStrToCsInit u0 = (sdtoint(hd u0) % 2 , tl u0)
	
  cStrToCsStep :: (Nat -> ((Rational, Str) -> (Rational, Str)))
  cStrToCsStep n0 (a,u0)= (a + sdtoint(hd u0) % (((2 ^ n0) * 2) * 2) ,
                           tl u0)
 
  cStrToCs :: (Str -> Rea)
  cStrToCs u0 = (\ n1 -> (fst (natRec n1 (cStrToCsInit u0) cStrToCsStep)) ,
                 id).
\end{verbatim}
Here \texttt{cStrToCsInit} corresponds to the first premise and \texttt{cStrToCsStep} to the second premise of $\DC$ in the proof. $\DC$ itself only appears as the \texttt{natRec} term in \texttt{cStrToCs}.
Informally we can represent the computational content by
\[
d_0d_1d_2\dots \mapsto \left(\left(\sum_{i=1}^{n}d_i2^{-i}\right)_n,\iota\right),
\]
where $\iota:\mathbb{Z}^+\to\mathbb{N}$ is the canonical inclusion.

For the converse of Theorem \ref{thm:StrToCs} we first prove the following two lemmas.
\begin{lem}[Special case of \texttt{ApproxSplit}]
	\label{thm:ApproxSplit}
	Let $a,b:\mathbb{Q}$ then
	\[
		\forall_{a,b\in\total,x\in\R}\left(a<b\to a\leq x\vee x\leq b\right).
	\]
\end{lem}
\begin{proof}
	Given $(as,M)$ a Cauchy-sequence converging to $x$. Find $p$ be such that $\frac{1}{2^p}<b-a$ (which is possible since $a,b\in\total$).
	Then for $n\geq M\ p$
	\[
		|x-(as\ n)|\leq  2^{-(p+1)},
	\]
	i.e.~$x\in\left[(as\ n)-\frac{1}{2^{p+1}},(as\ n)+\frac{1}{2^{p+1}}\right]$. Case $(as\ n)\leq \frac{a+b}{2}$. In that case $x\leq (as\ n)+\frac{1}{2^{p+1}}< b$. Otherwise $a\leq x$.
\end{proof}
Note that it can be proven more generally for $y,z\in\R$ and $y<z$ instead of $a,b$, but we will actually only need it for for the special cases $0<\frac 12$ and $-\frac 12 <0$. The extracted term for the former case is:
\begin{verbatim} 
   cApproxSplitZeroPtFive :: (Rea -> Bool)
   cApproxSplitZeropPtFive (RealConstr as m) = (as (m 3)) <= (1/4)
\end{verbatim}
where \texttt{cRatLeAbsBound} is the extracted term of a proof of $\forall_{a\in \total}\exists_{n\in\total} |a|\leq 2^n$.\\
For the converse of Theorem \ref{thm:StrToCs} we first prove the following lemma.
\begin{lem}[\texttt{CsToStrAux}] For all $x\in\R$ with $|x|\leq 1$
	\[
		\exists_{d\in\mathbf{SD},y\in\R}\left(|y|\leq 1 \wedge x=\frac{y+d}{2}\right).
	\]
\end{lem}
\begin{proof}
	Let $(as,M)$ a Cauchy-sequence converging to $x$.
	We use Lemma \ref{thm:ApproxSplit} with $0<\frac{1}{2}$ respectively $-\frac{1}{2}<0$. We define 
	\[
		d:=
		\begin{cases}
			\overline{1} &\text{if\quad} x<0,\\
			0 &\text{if\quad} -\frac{1}{2}<x<\frac{1}{2},\\
			1 &\text{if\quad} 0<x,
		\end{cases}
	\]
	$bs\ n := 2(as\ n)-d$, $N p:=M(p+1)$ and $y=2x-d$. Then $(bs,N)$ is a Cauchy-sequence converging to $y$. Furthermore $x=\frac 12(y+d)$ and $|y|\leq 1$ by definition. 
\end{proof}
The extracted term is given by
\begin{verbatim}
  cCsToStrAux :: (Rea -> (Sd, Rea))
  cCsToStrAux (RealConstr as m) = 
    (if ((as (m 3)<=-1/4) then 
        (SdL , (RealConstr (\ n3 -> (((2) * (as n3)) + (1))) 
                           (\ p3 -> (m (p3 + 1))))) 
     else 
       (if ((as (m 3)<=1/4) then 
          (SdM , (RealConstr (\ n3 -> ((2) * (as n3))) 
                             (\ p3 -> (m (p3 + 1))))) 
        else 
          (SdR , (RealConstr (\ n3 -> (((2) * (as n3)) + (-1)))
                             (\ p3 -> (m (p3 + 1))))))))).
\end{verbatim}
Again we can represent the computational content informally, namely
\[
	(as,M)\mapsto (\mathtt{g}(as,M),(\mathtt{h}(as,M),N)),
\]
where $N\ p:=M(p+1)$ and $\mathtt{g},\mathtt{h}$ are the functions
\begin{align*}
	\mathtt{g}(as,M)&:=
	\begin{cases}
		\bar{1} &\text{if\quad}as (M\ 3)\leq -\frac{1}{4},\\
		0 &\text{if\quad}|as(M\ 3)|\leq\frac{1}{4},\\
		1 &\text{otherwise}.
	\end{cases}\\
	\mathtt{h}(as,M)&:=2as-g(as,M)
\end{align*}
Using this lemma the proof of the translation from Cauchy-sequences to stream is very short:
	\begin{thm}[\texttt{CsToStr}]
		\label{thm:cstostr}
		\[
			\forall_x\left(x\in\R\to |x|\leq 1\rightarrow x\in\coi\right)
		\]
	\end{thm}
	\begin{proof}
		Assume $x\in\R$ and $|x|\leq 1$. We
		use \eqref{eq:coind} with
			\[
				Px:= \exists_{d\in\mathbf{Sd},y\in\R}\left(|y|\leq 1 \wedge x=\frac{y+d}{2}\right).
			\]
		By the previous lemma it suffices to prove the second premise, namely
		\[
			\forall_x\left(P x \to \exists_{d\in\mathbf{Sd},x'}\left(x'\in\left({^{co}\textbf{I}}\cup P\right)\wedge|x'|\leq 1\wedge x=\dfrac{d+x'}{2}\right)\right).
		\]
		But this follows immediately by another application of the lemma, namely let $d\in\sd,y\in \R$ with $|y|\leq 1$ and $x=\frac{y+d}{2}$. Then by the lemma there are $e\in\sd,z\in\R$ with $|z|\leq 1$ and $y=\frac{z+e}{2}$ and hence $y\in P\subseteq \coi\cup P$.
	\end{proof}
The extracted term is
\begin{verbatim}
  cCsToStr :: (Rea -> Str)
  cCsToStr x0 = strCoRec 
                 (cCsToStrAux x0) 
                 (\ sx1 -> (case sx1 of { (,) s0 x0 -> 
                 	          (case (cStrToCsAux x0) of { (,) s1 x1 -> 
                 	          	(s0 , (Right (s1 , x1))) }) })),
\end{verbatim}
and informally
\[
	x\mapsto \pi_0(\texttt{cCsToStrAux}\ x)::\texttt{cCsToStr}(\pi_1(\mathtt{cCsToStrAux}\ x)).
\]
\section{Convergence theorem}\label{ConvThe}
The convergence theorem states that the signed digit representation is closed under limits.
\noindent
In this section we consider a direct proof of this theorem, which only relies on the signed digit representation of real numbers, and an indirect proof, which works with Cauchy reals and uses the translation between the signed digit code and Cauchy reals. After proving the convergence theorem in these two ways, we compare the extracted terms of both proofs.
\begin{defi}[Convergence]
	Let $xs:\mathbb{N}\to\mathbb{R}$ and $M\in\mathbf{Mon}$ then $xs$ is a \textit{Cauchy-sequence} with modulus $M$ iff
	\[
	\forall_{p\in\total}\forall_{n,m\geq M(p)} |(xs\ n)-(xs\ m)|\leq 2^{-p},
	\]
	we also write $\mathbf{Cauchy}(xs,M)$.
	The sequence $xs$ converges to $x$ with Modulus $M$ iff
	\[
	\forall_{p\in\total}\forall_{n\geq M(p)} |x-(xs\ n)|\leq 2^{-p},
	\]
	we write $\mathbf{Conv}(xs,x,M)$.
\end{defi}
The convergence theorem can now be stated in the following way.
\begin{thm}[\texttt{SdLim}]\label{Thm:CoILim}
	Let $xs:\mathbb{N}\rightarrow \mathbb{R}$ be a sequence of reals in $\coi$ which converges to some real $x$ with modulus $M$. Then $x\in\coi$, i.e.
	\[
	\forall_{x,xs;M\in\mathbf{Mon}}\left(\forall_{n\in\total} (xs\ n)\in\coi\to\mathbf{Conv}(xs,x,M)\to x\in\coi\right).
	\] 
\end{thm}

\subsection{Direct approach}
The following approach was already considered in \cite{ediss28929} and is adjusted here to our setting.
\begin{lem}[\texttt{CoINegToCoIPlusOne}, \texttt{CoIPosToCoIMinusOne}]\label{Lem:CoINegToCoIPlusOne}
\begin{align*}
\forall_{x\in\coi} (x\leq 0 \rightarrow {^{co}\textbf{I}}(x+1))\\
\forall_{x\in\coi}( 0\leq x \rightarrow {^{co}\textbf{I}}(x-1))
\end{align*}
\end{lem}
\begin{proof}
Since the proofs are very similar, we only prove the first formula.
We use \eqref{eq:coind} with $Px:= \exists_{y\in\coi}\left( y\leq 0 \wedge y+1 = x \right)$. We need to prove the second premise,
namely
\[
\forall_x\left(Px \rightarrow\exists_{d\in\mathbf{Sd},x'}\left(x'\in\left({^{co}\textbf{I}}\cup P\right)\wedge|x'|\leq 1\wedge x=\dfrac{d+x'}{2}\right)\right)
\]
Let $x\in P$, $y\in{^{co}\textbf{I}}$ with $y\leq 0$ and $y+1=x$ be given. Our goal is 
\[
\exists_{d\in\mathbf{Sd},x'}\left(x'\in\left({^{co}\textbf{I}}\cup P\right)\wedge|x'|\leq 1\wedge x=\dfrac{d+x'}{2}\right).
\]
From $y\in{^{co}\textbf{I}}$ we get $e$ and $y'$ with $e\in\textbf{Sd}$, $y'\in\coi$, $|y'|\leq 1$ and $y=\frac{e+y'}2$. We make a case distinction on $\textbf{Sd}\ e$:\\
If $e = -1$, we define $d:=1\in\mathbf{Sd}$ and $x':=y'$. Then $|x'|\leq 1$ and $x'\in \coi$ by definition. Furthermore we have $$x=y+1=\frac{-1+y'}2+1 = \frac{1+y'}2=\frac{d+x'}{2}.$$
If $e=0$, we define $d:=1$ and $x':=y'+1$. In this case we prove $Px'$, namely we show $y'\in\coi$, $y'\leq 0$ and $x'=y'+1$. We only need to prove $y'\leq 0$ which follows directly from $y\leq 0$ and $y=\frac{0+y'}2$.\\
The last case is $e=1$. Because of $y\leq 0$, $y=\frac{-1+y'}{2}$ and $|y'|\leq 1$, this is only possible if $y$ is equal to $0$, and therefore $x=1$. Hence we define $d:=1$ and $x':=1$. Then  $x=\frac{d+x'}2$ and $x'=1\in\coi$ is easily proven by coinduction. (For details we refer to the Minlog implementation of the theorem \texttt{CoIOne} in \texttt{examples/analysis/sddiv.scm}.)
\end{proof}
A realizer of the first formula is a function \texttt{f}, which takes a signed digit stream of a real number $x$ and returns a signed digit stream of $x+1$ if $x\leq 0$. 
The extracted term of the proof of the first statement translated to Haskell is
\begin{verbatim}
  cCoINegToCoIPlusOne :: (Str -> Str)
  cCoINegToCoIPlusOne u0 = aiCoRec u0 
                             (\ u1 -> (case (hd u1) of 
		                                { SdR -> (SdR , (Left cCoIOne)) ;
			                                  SdM -> (SdR , (Right (tl u1))) ;
			                                  SdL -> (SdR , (Left (tl u1))) })),  
\end{verbatim}
where
\begin{verbatim}
  cCoIOne :: Str
  cCoIOne = (aiCoRec () (\ g -> (SdR , (Right ()))))
\end{verbatim}
is the stream representing $1$. Unfolding \texttt{strCoRec} once yields
\texttt{cCoIOne = C SdR cCCoIOne}, i.e.~it is a constant stream of \texttt{SdR}.

Another way to characterise this function $\texttt{f}$ is to give its computation rules:
\begin{align*}
\texttt{f} (\mathrm{C}\overline{1}v) &:= \mathrm{C}1v \\
\texttt{f} (\mathrm{C}0v) &:= \mathrm{C}1(\texttt{f}v)  \\
\texttt{f} (\mathrm{C}1v) &:= [1,1,\dots] 
\end{align*}
Analogously as extracted term of the second statement of this lemma, we get a function $\texttt{g}:\texttt{Str}\rightarrow\texttt{Str}$ which is characterised by the rules
\begin{align*}
\texttt{g} (\mathrm{C}\overline{1}v) &:= [\overline{1},\overline{1},\dots] \\
\texttt{g} (\mathrm{C}0v) &:= \mathrm{C}\overline{1}(\texttt{g}v)  \\
\texttt{g} (\mathrm{C}1v) &:= \mathrm{C}\overline{1}v.
\end{align*}
It takes a signed digit stream of a real $x$ and returns a signed digit stream of $x-1$ if $0\leq x$.\\
Using this lemma, we are now able to prove the following lemma:
\begin{lem}[\texttt{CoIToCoIDouble}]\label{Lem:CoIDouble}
\begin{align*}
\forall_{x\in\coi} \left(|x|\leq \frac 12 \rightarrow 2x\in\coi\right)
\end{align*}
\end{lem}
\begin{proof}
We apply $\coi^-$ and get $d\in\sd$, $x'\in\coi$ with $|x'|\leq 1$ and $x=\frac{d+x'}{2}$. We distinguish cases on $d\in\sd$:\\
$d=1$: Then $2x-1=x'$ and $|x|\leq \frac 12$ which imply $x'\leq 0$. By the first part of Lemma \ref{Lem:CoINegToCoIPlusOne} $1+x'=2x\in\coi$.\\
$d=-1$: As the first case but with the second part of Lemma \ref{Lem:CoINegToCoIPlusOne}.\\
$d=0$: In this case $2x=x'$ and $x'\in\coi$ by assumption.
\end{proof}
In Haskell notation the extracted term is given by
\begin{verbatim}
  cCoIToCoIDouble :: (Str -> Str)
  cCoIToCoIDouble u0 = case (hd u0) of
      { SdR -> (cCoINegToCoIPlusOne (tl u0)) ;
        SdM -> (tl u0) ;
        SdL -> (cCoIPosToCoIMinusOne (tl u0)) }
	
\end{verbatim}
Again we give a more readable characterisation of the extracted term \texttt{D} by the computation rules
\begin{align*}
\texttt{D}(\mathrm{C}\overline{1}u)&:=\texttt{g}u\\
\texttt{D}(\mathrm{C}0u)&:=u\\
\texttt{D}(\mathrm{C}1u)&:=\texttt{f}u,
\end{align*}
where \texttt{f}, \texttt{g} are the computational content of the previous lemma.
The following lemma is a special case of 
\[
\forall_{x\in\coi,y\in\coi}\frac{x+y}{2}\in\coi.
\]
This theorem is implemented as the theorem \texttt{Average} in \texttt{examples/analysis/average.scm} of Minlog and was considered in \cite{BergerSeisenberger12,MiyamotoSchwichtenberg15}. But here we give a direct proof of a special case because it is instructive and elementary.
\begin{lem}[Special case of \texttt{CoIAverage}]\label{Lem:CoIAvHalf}
\begin{align*}
\forall_{x\in\coi}\left(\frac x2 \pm \frac 14 \right)\in\coi
\end{align*}
\end{lem}
\begin{proof}Applying
$\coi^-$ yields $x'\in \coi$ and $d \in \textbf{Sd}$ with $x=\frac{d+x'}{2}$. 
We show only ${^{co}\textbf{I}}\left( \frac x2 + \frac 14 \right)$, the other case is proven analogously. We distinguish cases on $d\in\sd$: \\
$d=1$: Then $ \frac x2 +\frac 14 = \frac{2+x'}4 = \frac 12(1+\frac{x'}2) $. \\
$d=0$: Then $\frac x2 +\frac 14 = \frac {1+x'}4 = \frac 12\frac{1+x'}2$.\\
$d=-1$: Then $\frac x2 +\frac 14=\frac {x'}4=\frac 12\frac {x'}2$.\\
In each case we apply Lemma \ref{Lem:CoIClosure} twice to get $\left(\frac x2 +\frac 14\right)\in\coi$
\end{proof}
We denote the extracted term of the proven statement by $q^+$. From the proof and the fact that the extracted term of Lemma \ref{Lem:CoIClosure} is given by $\mathrm{C}$, one easily sees that $q^+$ has the following computation rules:
\begin{align*}
q^+(\overline{1}u)&:= 0 0 u\\
q^+(0u)&:= 0 1 u\\
q^+(1u)&:= 1 0 u
\end{align*}
Analogously, the extracted term $q^-$ of the statement $\forall_x.{^{co}\textbf{I}}x  \rightarrow {^{co}\textbf{I}}\left( \frac x2 - \frac 14 \right)$ is characterised by
\begin{align*}
q^-(\overline{1}u)&:= \overline{1} 0 u\\
q^-(0u)&:= 0 \overline{1} u\\
q^-(1u)&:= 0 0 u.
\end{align*}
In the direct proof of \texttt{sdlim} below we will make use of the following case-distinction. To shorten the extracted term we outsource this case-distinction into a separate lemma.
\begin{lem}[\texttt{TripleCases}]
	\label{lem:Triple}
	For $x\in\coi$
	\[
		x\in \left[\dfrac 18,1\right]\vee
		x\in \left[-1,-\frac 18\right]\vee 
		x\in \left[-\frac 14,\frac 14\right]
	\]
\end{lem}
\begin{proof}
	Triple application of ${^{co}\textbf{I}}^-$ to $x\in\coi$ gives $d_1,d_2,d_3\in \textbf{Sd}$ and $y'\in {^{co}\textbf{I}}$ such that $x=\frac{4d_1+2d_2+d_3+y'}8$. The claim follows by case-distinction on $d_1,d_2$ and $d_3$. Namely writing $x=d_1d_2d_3y'$ we have
	\[
	\begin{drcases}
		11d_3y'\\
		10d_3y'\\1\overline{1}1y'\\
		1\overline{1}0y'\\
		011y'\\
		010y'
	\end{drcases}\to  \frac{1}{8}\leq x\qquad
	\begin{drcases}
		\overline{11}d_3y'\\
		\overline{1}0d_3y'\\
		\overline{1}1\overline{1}y'\\
		\overline{1}10y'\\
		0\overline{11}y'\\
		0\overline{1}0y'
	\end{drcases}\to x\leq -\frac 18\qquad
	\begin{drcases}
		00d_3y'\\
		\overline{1}11y'\\
		1\overline{11}y'\\
		01\overline{1}y'\\
		0\overline{1}1y'
	\end{drcases}\to -\frac 14\leq x\leq \frac 14.
	\]
\end{proof}
We omit the extracted term in Haskell here since it is quite long and unreadable due to the $17$ case-distinctions. The computational content is basically the diagram in the proof above.

With these preparations we are now able to give the direct proof of Theorem \ref{Thm:CoILim}.
\begin{proof}[Proof.(\,$\mathtt{SdLim}$, direct)]
We show that
\begin{align*}
\forall_x\left(\exists_{xs;M\in\textbf{Mon}}\left(\forall_{n\in \total}(xs\ n)\in\coi\wedge \forall_{p\in\total}\forall_{n\geq Mp}|x-(xs\ n)|\leq 2^{-p} \right)\rightarrow {^{co}\textbf{I}}x\right),
\end{align*}
which is equivalent. We use \eqref{eq:coind} with $P$ the premise of the formula above:
\[
Px := \exists_{xs;M\in\textbf{Mon}}\left(\forall_{n\in \total}(xs\ n)\in\coi\wedge \forall_{n\geq M(p)}|x-(xs\ n)|\leq 2^{-p} \right)
\]
Again, we need to prove the second premise, namely
\[
\forall_x\left(P x\to \exists_{d\in\sd,x'}\left(x'\in\left({^{co}\textbf{I}}\cup P\right)\wedge|x'|\leq 1\wedge x=\dfrac{d+x'}{2}\right)\right)
\]
So let $x$, $xs$ and $M\in \mathbf{Mon}$ be given and assume $\forall_{n\in \textbf{T}}(xs\ n)\in\coi$ and $\forall_{p\in \total}\forall_{n\geq M(p)}|(xs\ n)-x|\leq 2^{-p}$. We use the lemma above with $(xs (M 4))\in\coi$ and get three cases: 

(i) $\frac 18\leq xs(M(4))$. In this case we choose 
\[
d:=1\ \text{and}\ x':=2x-1.
\]
Then $|x'|\leq 1$ and $x=\frac{d+x'}2$ follow directly. We show that $Px'$, so  we define 
\[
ys\ n:= 2(xs((M(4))\sqcup n))-1,
\]
where $m\sqcup l := \max\{m,l\}$ and $N(p):=M(p+1)\in\mathbf{Mon}$.
The statement $\forall_{n\geq N p}|ys(n)-x'|\leq 2^{-p}$ is a direct consequence of $\forall_{n\geq M(p)}|xs(n)-x|\leq 2^{-p}$ and it remains to show $\forall_{n\in \total}(ys\ n)\in \coi$.
We calculate 
\[
ys\ n=4\left( \frac{xs(M(4)\sqcup n)}2 - \frac 14 \right),
\]
and conclude $\frac 14(ys\ n)\in\coi$ by Lemma \ref{Lem:CoIAvHalf}. Furthermore, we have $xs(M(4))\geq \frac 18$ and $\forall_{n\geq M(4)}|xs(n)-x|\leq \frac{1}{16}$ and therefore
 \begin{align*}
xs(M(4)\sqcup n) &= (xs(M(4)\sqcup n)-x)+(x -xs(M(4)))+xs(M(4)) \\ &\geq -\frac 1{16} -\frac 1{16} + \frac 18 =0.
\end{align*}
Hence $0 \leq xs(M(4)\sqcup n)\leq 1$, which implies $\left|\frac{xs(M(4)\sqcup n)}2 - \frac 14\right| \leq \frac 14$ and by double application of Lemma \ref{Lem:CoIDouble} we finally get $ys\ n\in{^{co}\textbf{I}}$. 

(ii) $xs(M(4))\leq -\frac 18$. In this case we define
\[
d:=-1,\quad x':=2x+1,\quad ys\ n:=(2xs(M(4)\sqcup n)+1),\quad Np:=M(p+1).
\] 
The proof in this case is analogous to the proof of the first case.

(iii) $-\frac{1}{4} \leq f(M(4)) \leq \frac 14$. We define 
\[
d:=0\quad \text{and}\quad  x':=2x.
\]
Again we show $P x'$, namely
\[
\exists_{ys;N\in\textbf{Mon}}\left(\forall_{n\in \total}(ys\ n)\in\coi\wedge \forall_{n\geq Np}|x'-(ys\ n)|\leq 2^{-p} \right).
\]
To this end we define
\[
ys\ n:=2xs(M(4)\sqcup n)\quad \text {and}\quad N p:=M(p+1).
\]
Again, the right side of the conjunction follows from the assumptions. For the left side consider
\[
	|2(ys\ n)|=|xs(M(4)\sqcup n)|\leq |xs(M(4)\sqcup n)-x| + |x - xs(M(4))| +|xs(M(4))|\leq \frac 12,
\]
which implies $(ys\ n)\in\coi$ by Lemma \ref{Lem:CoIDouble}. 
\end{proof}
The extracted term is
\begin{verbatim}
  coilim :: (((Pos -> Nat), (Nat -> Str)) -> Str)
  coilim (m,us0) = aiCoRec (m,us0) 
    (\ mus1 -> (case mus1 of
      { (,) m us -> (case (cTripleCases (us (m 4))) of
        { Left() -> (cSdLimCaseR m us) ; 
        	 Right(Left ()) ->	(cSdLimCaseL m us) ;
        	 Right(Right()) -> (cSdLimCaseM m us)}}.
 	\end{verbatim}
The terms \texttt{cSdLimCaseR},\texttt{cSdLimCaseL} and \texttt{cSdLimCaseM} are given by
\begin{verbatim}
  cSdLimCaseR :: ((Pos -> Nat) -> ((Nat -> Str) ->             
                   (Sd, (Either Str ((Pos -> Nat), (Nat -> Str))))))
  cSdLimCaseR m0 us1 = (SdR , 
    (Right ((\ p2 -> (m0 (p2 + 1))) , 
            (\ n2 -> (cCoIToCoIDoublePlusOne (us1 ((m0 3) + n2)))))))
	
  cSdLimCaseM :: ((Pos -> Nat) -> ((Nat -> Str) -> 
                   (Sd, (Either Str ((Pos -> Nat), (Nat -> Str))))))
  cSdLimCaseM m0 us1 = (SdM , 
    (Right ((\ p2 -> (m0 (p2 + 1))) , 
            (\ n2 -> (cCoIToCoIDouble (us1 ((m0 3) + n2)))))))
	
  cSdLimCaseL :: ((Pos -> Nat) -> ((Nat -> Str) -> 
                   (Sd, (Either Str ((Pos -> Nat), (Nat -> Ai))))))
  cSdLimCaseL m0 us1 = (SdL , 
    (Right ((\ p2 -> (m0 (p2 + 1))) , 
            (\ n2 -> (cCoIToCoIDoubleMinusOne (us1 ((m0 3) + n2)))))))
\end{verbatim}
In the following we will discuss the computational content, we will denote it by \texttt{Lim}, in more detail.
It has the type \begin{align*}
\texttt{Lim}:(\mathbb{Z}^+\rightarrow \mathbb{N})\rightarrow (\mathbb{N}\rightarrow \str)\rightarrow \str.
\end{align*}
It takes as inputs the modulus of convergence and the sequence of streams and returns the stream representing the limit. In order to give a more readable characterisation of \texttt{Lim}, we define the following sets
\begin{align*}
\textbf{R}&:= \{11v, 10v, 1\overline{1}1v, 1\overline{1}0v, 011v, 010v\mid v: \str \}\\
\textbf{M}&:=\{00v, \overline{1}11v, 1\overline{11}v, 01\overline{1}v, 0\overline{1}1v\mid v:\str \}\\
\textbf{L}&:= \{\overline{11}v, \overline{1}0v, \overline{1}1\overline{1}v, \overline{1}10v,0\overline{11}v, 0\overline{1}0v\mid v:\str\}.
\end{align*}
which correspond to the intervals from Lemma \ref{lem:Triple}.
According to the proof we then have the following rule for \texttt{Lim}:
\begin{align*}
\texttt{Lim}\ M\ F :=
\begin{cases}
     \mathrm{C}\ 1\ (\texttt{Lim}\ \lambda_pM(p+1)\ \lambda_n(\texttt{DD}q^-F(M(4)\sqcup n)))& \text{if } F(M(4))\in \textbf{R} \\
    \mathrm{C}\ 0\ (\texttt{Lim}\ \lambda_pM(p+1)\ \lambda_n(\texttt{D}F(M(4)\sqcup n)))& \text{if } F(M(4))\in \textbf{M} \\
    \mathrm{C}\ \overline{1}\ (\texttt{Lim}\ \lambda_pM(p+1)\ \lambda_n(\texttt{DD}q^+F(M(4)\sqcup n)))& \text{if } F(M(4))\in \textbf{L}
   \end{cases}
\end{align*}
The functions \texttt{D}, $q^+$ and $q^-$ are the computational content of the lemmas above.
Note that the definition of the new sequence is not unique. For reasons of efficiency one should be flexible with the choice of the new sequence, which is called $ys$ in the proof above. For example by choosing $ys$ one can replace $M(4) \sqcup n$ by $M(4)+n$. The efficiency depends on the concrete sequence. In the Minlog file we have chosen $M(4)+n$ because the proofs are simpler with the addition instead of the maximum.
\subsection{Indirect approach}
Now we redo the proof using translations between the Cauchy and sd-representation. First we state the completeness of Cauchy-reals, which will be used.
\begin{thm}[\texttt{RealComplete}]
	\label{thm:Completeness}
	Assume $xs$ and $M\in \mathbf{Mon}$ such that $\forall_{n\in\total} (xs\ n)\in\R$ and $\mathbf{Cauchy}(xs,M)$ then there exists $x\in\R$ such that $\mathbf{Conv}(xs,x,M)$.
\end{thm}
\begin{proof}
	We refer to Theorem 2.3 in \cite{Schwichtenberg06e}.
\end{proof}
The extracted term is:
\begin{verbatim}
  cRealComplete :: ((Nat -> Rea) -> ((Pos -> Nat) -> Rea))
  cRealComplete xs0 m1 = RealConstr 
                            (\ n -> (realSeq (xs0 n) 
                                       (realMod (xs0 n) (cNatPos n)))) 
                            (\ p -> ((m1 (p + 1)) `max` ((p + 1) + 1)))
\end{verbatim}
Note the following: Given witnesses $(as_n,M_n)_n$ to $\forall_{n\in\total}(xs\ n)\in\R$, the rational sequence witnessing the limit is given by $as\ n:=as_n(M_n\ n)$ and the Cauchy-modulus of the limit-real is given by $N p:= \mathtt{max}(M(p+1),p+2)$ where $M$ is the modulus of convergence.

As preparation for the indirect proof we state some elementary properties of limits and sequences that we will need but do not have any computational content.
\newpage
\begin{lem}
	Assume $xs$ is a sequence of reals, $x$ is another real and $M\in\mathbf{Mon}$ such that $\mathbf{Conv}(xs,x,M)$. Then we have
	\begin{enumerate}[(i)]
		\itemsep=0em
		\item $\mathbf{Cauchy}(xs,N)$, where $N p:=M(p+1)$,
		\item $\forall_{n\in \total} |(xs\ n)|\leq 1 \to |x|\leq 1$ and 
		\item $\forall_{y,N\in\mathbf{Mon}}(\mathbf{Conv}(xs,y,N) \to x=y)$.
	\end{enumerate}
\end{lem}
\begin{proof}
	(i) Follows directly by using the triangle inequality and the definitions.\\
	(ii) For arbitrary $p\in \total$ let $n:=M(p)$. Then $|x|\leq |x-xs(n)|+|xs(n)|\leq 2^{-p}+1$. As $p\in \total$ is arbitrary it follows $|x|\leq 1.$\\
	(iii) We have $xs(n)-x$ converges to $0$ with modulus $M$ and $xs(n)-y$ converges to $0$ with modulus $N$. Therefore, $x-y$ converges to zero with modulus $p\mapsto \max\{M(p+1),N(p+1)\}$. I.e.~$|x-y|\leq 2^{p}$ for all $p$ and therefore $x-y=0$.
\end{proof}

Using the lemmas above, the indirect proof of the convergence theorem becomes quite short:
\begin{proof}[Proof.(\,$\mathtt{SdLim}$, indirect)]
Assume $x$, $xs$, $M\in \mathbf{Mon}$, $\forall_{n\in \textbf{T}}(xs\ n)\in\coi$ and $\forall_{n\geq M(p)}|(xs\ n)-x|\leq 2^{-p}$. We apply Theorem \ref{thm:StrToCs} to $\forall_{n\in\total}(xs\ n)$ and get
\[
\forall_{n\in\total} xs\in\R\wedge |xs\ n|\leq 1.
\]
By the lemma above $xs$ is a Cauchy-sequence. So we apply 
Theorem \ref{thm:Completeness} to get $y\in\R$ with $\mathbf{Conv}(xs,y,N)$. By the above lemma $x=y\in\R$ and $|y|\leq 1$, so $x\in\coi$ by Theorem \ref{thm:cstostr} and \ref{Lem:CoICompat}.
\end{proof}
\noindent The extracted term for this proof is given by
\begin{verbatim}
  cCoILim :: ((Pos -> Nat) -> ((Nat -> Str) -> Str))
  cCoILim m g =  cCsToStr (cRealComplete 
                             (\ n -> (cStrToCs (g n))) 
                             (\ p -> (m (p + 1))))),
\end{verbatim}
i.e.~given a sequence $u_0u_1\dots$ of $\sd$-streams we apply $\mathtt{cRealComplete}$ to the sequence of translated stream $\mathtt{cStrToCs}(u_0)\mathtt{cStrToCs}(u_1)\dots$ and then translate the result back.
\subsection{Comparison} We now compare the two algorithms obtained by the direct and indirect method. To understand the results of the runtime-experiments we analyze the \textit{lookahead} of the algorithms first. Both limit algorithms have a sequence $F:\mathbb{N}\to\mathtt{Str}$ of streams and a modulus $M:\mathbb{Z}^+\to \mathbb{N}$ of convergence as inputs and they produce one output-stream. Here, the lookahead for some $n\in\mathbb{N}$ is given by two natural numbers $m_0,m_1\in\mathbb{N}$. Namely, to compute the first $n$ output digits we need the first $m_0$ digits of the first $m_1$ elements of $F$.

Unfolding $\mathtt{cRealComplete}$ in the definition of the indirect case leads to
\[
	\mathtt{cCoILim}(M,F)= \mathtt{cCsToStr} \left(\left(\sum_{i=1}^{n}\frac{(F(n))_i}{2^i}\right)_n,\lambda_pM(p+2)\right),
\]
where $\mathtt{cCsToStr}(as,M)$ compares $as(M\;3)$ with $\pm 0.25$. Hence, to compute the $n$-th digit of $\mathtt{cCoILim}(M,F)$ we need to examine the first $M(n+4)$ digits of $F(M(n+4))$. The algorithm in the direct case was given by
\begin{align*}
	\texttt{Lim}\ M\ F :=
	\begin{cases}
		\mathrm{C}\ 1\ (\texttt{Lim}\ \lambda_pM(p+1)\ \lambda_n(\texttt{DD}q^-F(M(4)\sqcup n)))& \text{if } F(M(4))\in \textbf{R} \\
		\mathrm{C}\ 0\ (\texttt{Lim}\ \lambda_pM(p+1)\ \lambda_n(\texttt{D}F(M(4)\sqcup n)))& \text{if } F(M(4))\in \textbf{M} \\
		\mathrm{C}\ \overline{1}\ (\texttt{Lim}\ \lambda_pM(p+1)\ \lambda_n(\texttt{DD}q^+F(M(4)\sqcup n)))& \text{if } F(M(4))\in \textbf{L}
	\end{cases}
\end{align*}
By examining the defining equations of $\mathtt{D},q^\pm$ one easily sees that all these functions needs at most the first $n+1$ digits of the input stream to compute the first $n$ digits of the output stream. For the first digit we need to decide whether $F(M(4))$ is in $\textbf{R},\textbf{M}$ or $\textbf{L}$ which requires the first three digits of $F(M(4))$, since we apply $\mathtt{cCoITripleClosure}$.  All in all, to compute the $n$-th digit of $\texttt{Lim}\ M\ F$ we need to examine the first $3n$ digits of $F(M(n+3))$. This follows as $M$ is monotone and hence $M(4)\sqcup \dots \sqcup M(n+3)=M(n+3)$.

As we can see, in the direct case the lookahead depends in a way linearly on the modulus $M$ of convergence, whereas in the indirect case it depends quadratically on $M$. Furthermore, if the modulus of convergence is asymptotically lower than $\lambda_n 3n$, the indirect algorithm should outperform the direct one.

As a first test we run both algorithms in Haskell on the constant sequence
$F:=\lambda_n u_0$ which converges with the constant modulus $\lambda_p 0$. Here $u_0$ is a pseudo-random stream of $\mathbf{Sd}$  generated with the Haskell $\mathtt{System.Random}$ package. To test the dependence of the two algorithms on the modulus of convergence we artificially set different moduli and compute different amounts of digits. All measurements are the average for $n=10$ tries with different random numbers in seconds.
\begin{figure}[ht]
	\begin{tabular}{c|ccc}
		Mod & 50 digits & 100 digits & 200 digits \\
		\hline
		$\lambda_pp$& 1.78 & 12.2 & 87\\
		$\lambda_pp^2$& 1.84 & 12.7 & 90\\
		$\lambda_pp^3$& 2.25 & 13.4 & 95\\
	\end{tabular}
\caption{First test - Constant sequences with different moduli for the direct algorithm}
\end{figure}
\begin{figure}[ht]
	\begin{tabular}{c|cccc}
		Mod & 10 digits & 20 digits & 50 digits & 100 digits \\
		\hline
		$\lambda_pp$& - & $<$ 0.05  & 0.075 & 0.21 \\
		$\lambda_pp^2$& 0.084  & 0.41 & 16.38 & 1140\\
		$\lambda_pp^3$& 4.3 & 503 & $>$1500 & -\\
	\end{tabular}\\
	\caption{Constant sequences with different moduli for the indirect algorithm}
\end{figure}\\
As a second experiment we take the geometric series $(x^n)_n$ for some $|x|\leq 0.5$. This is a Cauchy-sequence converging to $0$ with modulus $\lambda_pp$, since for $n\leq m$ and $|x|\leq0.5$ we have
\[
	|x^n-x^m|\leq |x^n||1-x^{n-m}|\leq \frac{1}{2^n}.
\] 
Again we generate pseudorandom sequences $u$ and here we put a $0$ in front to ensure that the absolute value is bounded by $0.5$. Then we run both algorithms for the sequence $F$ given by
\[
	F\ 0 := 0::u\qquad F(n+1)=\mathtt{cCoIMult}(0::u)(F\,n),
\]
where $\mathtt{cCoIMult}$ is the algorithm from \cite{schwichtenberg2021sdmult}. The results below are the average over $n=15$ tests. The direct algorithm did not terminate in a reasonable amount of time ($\leq 30$ minutes) for $n\geq30$ digits. As expected the indirect algorithm is better here, since the modulus of convergence is the identity here.
\begin{figure}[ht]
	\begin{tabular}{c|cc}
		digits & indirect & direct \\
		\hline
		5 & 0.74 & 0.69 \\
		10 & 3.3 & 23.4 \\
		20 & 26 & 1227 \\
		30 & 87 & $>$1500 \\
		40 & 239 & - \\
		50 & 502 & - \\
	\end{tabular}
	\caption{Second test - Geometric sequence}
\end{figure}

\section{Applications}
\label{sec:app}
\subsection{Heron's method}
To show an application of the two algorithms extracted in the last section, we define the Heron sequence and show that it converges to the square root.
\begin{defi}[Heron]\label{Def:Heron}
We define $\text{H}:\mathbb{R}\rightarrow \mathbb{N}\rightarrow \mathbb{R}$ by the computation rules
\[
\text{H}(x,0) := 1,\qquad
\text{H}(x,n+1):= \frac 12 \left( \text{H}(x,n)+\frac{x}{\text{H}(x,n)}\right).
\]
For every non-negative $x$ the sequence $H(x,\cdot) =: H(x):\mathbb{N}\rightarrow \mathbb{R}$ is the sequence, we get from Heron's method with initial value 1. Note that $H$ is well-defined for non-negative $x$ since $H(x,n)\geq 2^{-n}$.
\end{defi}
\begin{lem}\label{Lem:HeronConv}
For every $x\in [0,1]$ $H(x)$ converges to $\sqrt x$ with modulus $\iota:\mathbb{Z}^+\to \mathbb{N}$. Furthermore we have that
\[
\forall_{n\in\textbf{T}}\sqrt{x}\leq H(x,n).
\]
\end{lem}
\begin{proof}
Let $x\in [0,1]$ be given. We define $\Delta(x,n) := H(x,n)-\sqrt{x}$. We calculate
\begin{align*}
\Delta(x,n+1)&=\frac 12 \left( \text{H}(x,n)+\frac{x}{\text{H}(x,n)}\right) -\sqrt x = \frac{(H(n,x))^2 -2H(x,n)\sqrt x +x}{2H(x,n)}\\
&= \frac{(\Delta(x,n))^2}{2H(x,n)}.
\end{align*}
By induction on $n$ we immediately get $0\leq H(x,n)$ and therefore $0\leq \Delta(x,n+1)$. Since $\Delta(x,0)=1-\sqrt x\geq 0$ we have $\forall_{n\in\textbf{T}} \sqrt{x}\leq H(x,n)$.\\
Furthermore, we calculate:
\begin{align*}
\Delta(x,n+1)&= \frac{(\Delta(x,n))^2}{2H(x,n)}=\frac 12 \Delta(x,n)\frac{\Delta(x,n)}{H(x,n)}=\frac 12 \Delta(x,n)\left(1-\frac{\sqrt x}{H(x,n)}\right)\\
&\leq \frac{1}{2} \Delta(x,n)
\end{align*}
Therefore, by induction we have $|H(x,n)-\sqrt x|=\Delta(x,n)\leq 2^{-n}$ and this implies
\[
\forall_{p\in\textbf{T}}\forall_{n\geq p}|H(x,n)-\sqrt x|\leq 2^{-p},
\]
i.e.~$H(x)$ converges to $\sqrt x$ with modulus $\iota$.
\end{proof}
This lemma by itself does not have any computational content, but it states that $\iota$ is a modulus of convergence of $H x$ to $\sqrt{x}$. In some special cases we can improve on the modulus.
\begin{defi}[Poslog]
	For a positive integer $p$ we define $\texttt{poslog}(p)$ as the least natural number $n$ with $p\leq 2^n$. Equivalently it is the number of digits in the binary representation.
\end{defi}
One possibility to implement the function \texttt{poslog} is to define an auxiliary function $\texttt{auxlog}:\mathbb{Z}^+\rightarrow\mathbb{N}\rightarrow \mathbb{N}$ with the computation rules
\begin{align*}
	\texttt{auxlog}\ p\ n := 
	\begin{cases}
		n,& \text{if } p\leq 2^n \\
		\texttt{auxlog} \ p \ (n+1),& \text{otherwise,}   
	\end{cases}
\end{align*}
and then set $\texttt{poslog}(p) := \texttt{auxlog}\ p\ 0$.

\begin{prop}\label{Lem:PosLog}
If $x\in[\frac{1}{4},1]$ then $\mathtt{poslog}:\mathbb{Z}^+\rightarrow \mathbb{N}$ is a modulus of convergence of $H(x)$ to $\sqrt{x}$.
\end{prop}
\begin{proof}
Let $x\in [\frac 14,1]$. From Lemma \ref{Lem:HeronConv} we know $\forall_{n\in\total}\sqrt{x}\leq H(x,n)$ and therefore $\forall_{n\in\total}\frac 12\leq H(x,n)$. In the proof of Lemma \ref{Lem:HeronConv} the formula
\[
\Delta(x,n+1)= \frac{(\Delta(x,n))^2}{2H(x,n)}
\]
is proven. This implies $\Delta(x,n+1)\leq (\Delta(x,n))^2$. Since $\frac 12 \leq\sqrt{x}$, by induction we get
\[
\Delta(x,n)\leq 2^{-2^n}
\]
for all $n\in\total$. Hence for given $p$ and $n\geq \texttt{poslog}(p)$ we get $p\leq 2^n$ and 
\[
\left|H(x,n)-\sqrt x\right|=\Delta(x,n)\leq 2^{-2^n}\leq2^{-p}.\qedhere
\]
\end{proof}
\begin{lem}\label{Lem:HeronCoI}
For all $x \in \coi$ with $\frac 1 {16} \leq x$ we have $\forall_{n\in\textbf{T}}H(x,n)\in \coi$. Expressed as a formula
\begin{align*}
\forall_{x\in\coi}\left(\frac 1{16}\leq x \rightarrow \forall_{n\in\total}H(x,n)\in\coi\right).
\end{align*}
\end{lem}
\begin{proof}
We use the results  of \cite{MiyamotoSchwichtenberg15}, \cite{schwichtenberg2020sddiv} and of Section 3.3 from \cite{Masterarbeit}. Namely we have
\begin{align}\label{Form:CoIAv}
\forall_{x\in\coi,y\in\coi} \frac {x+y}2\in\coi.
\end{align}
In Minlog this theorem is implemented in \texttt{average.scm} in the folder \texttt{examples/analysis} and has the name \texttt{CoIAverage}. 
Furthermore
\begin{align}\label{Form:CoIDiv}
\forall_{x\in\coi,y\in\coi}\left(|x| \leq y \rightarrow \frac 1 4 \leq y \rightarrow \frac xy\in\coi\right)
\end{align}
is proven there. In Minlog this theorem is implemented in \texttt{sddiv.scm} in the folder \texttt{examples/analysis} and has the name \texttt{CoIDiv}.
Using these, the proof of this lemma is done by induction on $n$:
For $n=0$ it is easy since $H(x,0)=1$ and $1\in\coi$.
For any total $n$ we have 
\[
H(x,n+1)=\frac 12 \left( H(x,n)+\frac{x}{H(x,n)} \right).
\]
By Lemma \ref{Lem:HeronConv} we get $\sqrt{x}\leq H(n,x)$ and therefore 
\[
\sqrt{\frac 1{16} } =\frac 14 \leq H(x,n)\quad\text{and}\quad x\leq \sqrt{x} \leq H(x,n).
\]
Additionally, by the induction hypothesis, $H(x,n)\in\coi$. By (\ref{Form:CoIDiv}) we have $\frac{x}{H(x,n)}\in\coi$, so with (\ref{Form:CoIAv}) we get $H(x,n+1)\in\coi$.
\end{proof}
By \texttt{cCoIAverage} and \texttt{cCoIDiv} we denote the computational content of (\ref{Form:CoIAv}) and (\ref{Form:CoIDiv}). Each of these terms takes two streams of reals and returns a stream of their average and their quotient, respectively. Then the extracted term is
\begin{verbatim}
  cCoIHeron :: (Str -> (Nat -> Str))
  cCoIHeron u0 n1 = natRec n1 
                      cCoIOne 
                      (\ n2 -> (\ u3 -> (cCoIAverage u3 (cCoIDiv u0 u3))))
\end{verbatim}
Informally the computational content \texttt{Heron} is defined by recursion:
\begin{align*}
&\texttt{Heron}\ v\ 0 := [1,1,\dots]\\
&\texttt{Heron}\ v\ (n+1):= \texttt{cCoIAverage} (\texttt{Heron}\ v\ n) (\texttt{cCoIDiv}\ v\ (\texttt{Heron}\ v\ n))
\end{align*}
Which is Definition \ref{Def:Heron} in the notation of streams.
\begin{thm}\label{Thm:Sqrt}
\begin{align*}
\forall_{x\in\coi}\left(0\leq x \to \sqrt x\in\coi\right)
\end{align*}
\end{thm}
\begin{proof}
We apply \eqref{eq:coind} with
\[
Px := \exists_{y\in\coi}\left(0\leq y\wedge \sqrt y=x\right).
\]
To show the goal formula, we need to show the second premise, namely for all $x$
\[
\exists_{y\in\coi}\left(0\leq y\wedge \sqrt y=x\right) \rightarrow \exists_{d\in\sd,x'}\left(x'\in\left({^{co}\textbf{I}}\cup P\right)\wedge|x'|\leq 1\wedge x=\dfrac{d+x'}{2}\right).
\]
Let $y\in\coi$ with $0\leq x$ and $\sqrt y =x$ be given. Triple application of ${^{co}\textbf{I}}^-$ to $y\in\coi$ yields $d_1,d_2,d_3\in \textbf{Sd}$ and $y'\in {^{co}\mathbf{I}}$ with $y=d_1d_2d_3y'$. We distinguish three different cases:

If $y$ has one of the forms $\overline{1}d_2d_3y'$, $0\overline{1}d_3y'$ or $00\overline{1}y'$ we have $y\leq 0$ and therefore $x=\sqrt y=0$. Hence we define $d:=0$ and $x':=0$ and the claim follows immediately.

If $y$ has one of the forms $000y'$, $001y'$, $01\overline{1}y'$ or $1\overline{11}y'$ we can rewrite $y=00ey'$ for some $e\in \{0,1\}$. Here we define $d:=0$ and $x':= \sqrt{\frac{e+y'}2}$. Then \[
x=\sqrt{y}=\sqrt{\frac{e+y'}8}=\frac{\sqrt{\frac{e+y}2}}2=\frac{d+x'}2.
\]
Furthermore $\frac{e+y'}2\in\coi$ by Lemma \ref{Lem:CoIClosure} and  $0\leq\frac{e+y'}2$ since $0\leq y= \frac{e+y'}{8}$. Altogether we get $Px'$.

The remaining case is that $y$ has one of the forms $010y'$, $011y'$, $1\overline{1}1y'$, $1\overline{1}0y'$, $10d_3y'$ or $11d_3y'$. In that case we have $\frac 18 \leq y$. Hence by Lemma  \ref{Lem:HeronCoI} $\forall_{n\in\textbf{T}} {^{co}\textbf{I}}(H(y,n))$  and we know that $H(y)$ converges to $\sqrt{y}$ with modulus $\iota:\mathbb{Z}^+\rightarrow\mathbb{N}$ by Lemma \ref{Lem:HeronConv} . Thus we use Theorem \ref{Thm:CoILim} to get $x\in {^{co}\textbf{I}}$ and by one application of $\coi^-$
\[
\exists_{d\in\sd,x'\in\coi}\left(|x'|\leq 1\wedge x=\dfrac{d+x'}{2}\right),
\] which proves the goal formula.
\end{proof}
We omit the description of the extracted term as Haskell code as it is quite long due to the case distinctions. But we give an informal description of the extracted term:

By the definitions of \texttt{cSdLim} as the computational content of Theorem \ref{Thm:CoILim} and \texttt{Heron} as the computational content of Lemma \ref{Lem:HeronCoI} we have the following rules for the computational content $\texttt{sqrt}:\texttt{Str}\rightarrow \texttt{Str}$ of this theorem:  
\begin{align*}
\texttt{sqrt}(\overline{1}u) &:= [0,0,\dots]\\
\texttt{sqrt}(0\overline{1}u) &:= [0,0,\dots]\\
\texttt{sqrt}(00u) &:= 0\ \texttt{sqrt}\ u\\
\texttt{sqrt}(01\overline{1}u) &:= 0\ \texttt{sqrt}\ 1u\\
\texttt{sqrt}(1\overline{11}u) &:= 0\ \texttt{sqrt}\ 1u\\
\texttt{sqrt}\ u &:= \texttt{cSdLim}\ \iota \ (\texttt{Heron}\ u)
\end{align*}
The last rule shall only be applied if the other rules do not fit.
\subsection{Multiplication}
Our last application is motivated by Helmut Schwichtenberg and the Minlog file \texttt{sdmult.scm} in \texttt{examples/analysis}. There it is proven that for any $x,y\in\coi$ the product $xy$ is also in $\coi$.
In the following we use the limit-theorem to formulate another proof of this theorem. Our approach is based on repeated applications of $\coi^-$ to $y\in\coi$, namely
\[
xy=\frac{xd_1+xy_1}{2}=\frac{x(d_1+\frac{d_1}2)+x\frac{y_2}{2}}{2}=\dots=x\sum_{i=1}^{n}\frac{d_i}{2^i}+x\frac{y_n}{2^n},
\]
and the sequence $xs\ n:=\sum_{i=1}^{n}\frac{d_i}{2^i}$ converges to $y$.
In order to realize this idea we first define a constant $\mathtt{Sum}:\mathbb{L}(\mathbf{Sd})\to\mathbb{Q}$ by
\[
\mathtt{Sum}\,[]=0\qquad\mathtt{Sum}\,l :=\sum_{i=1}^{\mathtt{lth}(l)}\frac{(l)_i}{2^i},
\]
where $\mathtt{lth}:\mathbb{L}\to\mathbb{N}$ is the length-function for lists and $(l)_i$ is the $i$-th element of the list $l$. 
We prove the following.
\begin{lem}[\texttt{CoIMultSum}]
	\label{lem:multdecomp}Let $x\in\coi$ then
	for all $l:\mathbb{L}(\mathbf{Sd})$ in $\mathbf{T}$ we have
	\[
	x\cdot\mathtt{Sum}\,l\in\coi.		
	\]
\end{lem}
\begin{proof} We use the theorem \texttt{CoIAverage}, which was already mention in the proof of Lemma \ref{Lem:HeronCoI} and \texttt{CoISdTimes} (i.e.~$\forall_{x\in{\coi},d\in\sd}dx\in \coi$).
	
	The proof is done by induction on $l\in \total$. Namely if $l=[]$ then $x\cdot0=0\in\coi$ (see \texttt{CoIZero} in Minlog). Now assume that
	$
	x\cdot\mathtt{Sum}\,l\in\coi$ and
	$d\in\mathbf{Sd}$. We calculate
	\[
	x\cdot
	\mathtt{Sum}(d::l)=
	x\sum_{i=1}^{n+1}\frac{(l)_i}{2^i}=\frac 12\left(x\sum_{i=1}^{n}\frac{(l)_{i+1}}{2^i}+x\cdot (l)_1\right).
	\]
	Now we can
	apply \texttt{CoIAverage}, namely 
	$x\cdot l_1\in\coi$ by \texttt{CoISdTimes} and the first summand is in $\coi$ by the induction hypothesis. 
\end{proof}
Note that the induction hypothesis is applied to the list $(l)_2::\dots :: (l)_{n+1}$, so the stream computed in the previous step is not used. Therefore, the runtime of the algorithm must be at least quadratic in the length of the list in the output. We will later see that the runtime of the algorithm that is obtained is worse than the one extracted in \cite{schwichtenberg2021sdmult}. The Haskell-translation of the extracted term is
\begin{verbatim}
  cCoIMultSum :: Str -> [Sd] -> Str
  cCoIMultSum u0 l1 = listRec 
    l1 
    cCoIZero 
    (\ s2 -> (\ l3 -> (\ u4 -> (cCoIAverage (cCoISdTimes s2 u0) u4))))
\end{verbatim}
And the computational content of the lemma, here denoted $\mathtt{f}:\texttt{Str}\to\mathbb{L}(\sd)\to\texttt{Str}$, can also be represented in the more readable form by
\begin{align*}
	&\mathtt{f}(u,[]) = \mathtt{cCoIZero}\\
	&\mathtt{f}(u,d::l) = \mathtt{cCoIAverage}(
	\mathtt{f}(u,l),
	\mathtt{cCoISdTimes}(d,u)),
\end{align*}
where \texttt{cCoIZero} is analogous to \texttt{cCoIOne}:
\begin{verbatim}
  cCoIZero :: Str
  cCoIZero = (aiCoRec () (\ g -> (SdM , (Right ()))))	
\end{verbatim}
The next lemma is  
basically repeated application of $\coi^-$. The proof is very similar to the proof of Theorem \ref{thm:StrToCs} and so are the extracted terms.
\begin{lem}[\texttt{CoIToConvSeq}]
	\label{lem:multcoi}
	Let $x\in\coi$ then there exists $G:\mathbb{N}\to\mathbb{L}(\mathbf{Sd})$ in $\mathbf{T}$ such that
	\[
	\mathbf{Conv}(\lambda_n\mathtt{Sum}\,G(n),x,\iota).
	\]
\end{lem}
\begin{proof}
	For $x\in\coi$ we first show that
	\[
		\exists_{G\in\mathbf{T}}\forall_{n\in\mathbf{T}}\exists_{y\in\coi}\left(\mathtt{lth}(G\,n)=n\wedge x=\mathtt{Sum}\,(G\,n)+\frac{y}{2^n}\right).
	\]
	By application of $\mathbf{CC}$ it suffices to show
	\[
		\forall_{n\in\mathbf{T}}\exists_{l\in\mathbf{T}}\exists_{y\in\coi}\left(\mathtt{lth}(l)=n\wedge x=\mathtt{Sum}\,l+\frac{y}{2^n}\right),
	\]
	which is done by induction on $n\in\total$.
	If $n=0$ then choose $y:=x$ and $l=[]$.\\
	Now assume we have $l'\in\total$ with $\mathtt{lth}(l')=n$ and $z\in\coi$ with
	\[
	x=\mathtt{Sum}\,l'+\frac{z}{2^n}.
	\]
	We apply $\coi^-$ to $z$ and get $d\in\sd,y\in\coi$ with $z=\frac{y+d}{2}$, then $l:= l'\star d$ (Note that $\star$ denotes concatenation of lists) will do the trick.\\
	So assume we have $G$ with the properties above and $n\in\mathbf{T}$. Then there exists some $y_n\in\coi$ with $x=\mathtt{Sum}\,(G\,n)+\frac{y}{2^n}$ and
	\[
	\left|x-\mathtt{Sum}\,(G\, n)\right|\leq \frac{|y_n|}{2^n}\leq 2^{-n}.
	\]
	Hence $\mathtt{Sum}\,(G\, \cdot)$ converges to $x$ with modulus $\mathtt{PosToNat}$.
\end{proof}
The extracted term is
\begin{verbatim}
  cCoIToConvSeq :: Str -> Nat -> [Sd]
  cCoIToConvSeq u0 n1 = fst 
    (natRec n1 ([] , u0) 
    (\ n2 -> (\ g -> (case g of
      { (,) l u1 -> ((l ++ ((head u1) : [])) , (tail u1)) }))))
\end{verbatim}
Let $\mathtt{g}: \texttt{Str}\to\mathbb{N}\to\mathbb{L}(\sd)$ denote a simplified iterative version of the computational content of the last lemma. It can be given by the computation rules
\[
\mathtt{g}(u, 0) = ([],u)\qquad
\mathtt{g}(\mathrm{C}\,s\,v,n+1) = g(v,n)\star s,
\]
so this is the function that returns the first $n$ elements of the stream.
The last lemma we need states that limits are closed under multiplication, namely:
\begin{lem}
	\label{lem:limclosedmult}
	Assume $\mathbf{Conv}(ys,y,M)$ and $|x|\leq 1$ then $\mathbf{Conv}(\lambda_n (x\cdot(ys\,n)),xy,M)$.
\end{lem}
The proof is elementary.
Now we can apply the limit theorem.
\begin{thm} For all $x,y\in\coi$ we have $x\cdot y\in\coi$.
\end{thm}
\begin{proof}
	We apply Lemma \ref{lem:multcoi} to $y\in\coi$ in order to obtain $G\in\mathbf T$ with
	\[
		\mathbf{Conv}(\mathtt{Sum}(G,\cdot),y,\mathtt{PosToNat}).
	\]By Lemma \ref{lem:limclosedmult} $x\cdot(\mathtt{Sum}(G\,n))\in\coi$ for all $n$ and it converges to $x\cdot y$ by Lemma \ref{lem:limclosedmult}. Hence, we apply Theorem \ref{Thm:CoILim}. 
\end{proof}
The extracted term from Haskell is given by
\begin{verbatim}
  cCoIMultLim :: Str -> Str -> Str
  cCoIMultLim u0 u1 = cCoILim 
    id 
    (\ n2 -> (cCoIMultSum u0 (cCoIToConvSeq u1 n2))).
\end{verbatim}
Now let $\mathtt{Mult}:\texttt{Str}\to\texttt{Str}\to\texttt{Str}$ denote the computational content in human-readable. Then
\[
\mathtt{Mult} (u,v):= \mathtt{coilim}(\mathtt{PosToNat},\lambda_n \mathtt{f}(u,\mathtt{g}(v,n))).
\]
We compare the algorithms obtained in this way using the indirect and direct limit theorem with the algorithm from \cite{schwichtenberg2021sdmult} that was obtained from a direct proof. To this end we apply all three algorithms to two randomly generated sequences and measure the runtime in Haskell. The result is the average over $n=10$ tests.
\begin{figure}[ht]
	\begin{tabular}{c|ccc}
		number of digits & mult(schwicht) & mult(lim,indirect) & mult(lim,direct) \\
		\hline
		5 & 0.056 & 0.51 & 0.1 \\
		10 & 0.061 & 15 & 0.71 \\
		15 & 0.063 & 432 & 11.6 \\
		30 & 0.095 & - & 60.8 \\
		100 & 0.43 & - & - \\
	\end{tabular}
	\caption{Third test - Runtime of different multiplication algorithms}
\end{figure}
Although the algorithm using the direct limit is better here, the algorithm from \cite{schwichtenberg2021sdmult} still performs best.
It seems that in order to obtain efficient algorithms, completeness results should only be used if they are really needed. 
\section{Conclusion and further work}
	We presented a formal method for extracting verified algorithms for exact real number arithmetic using different representations. All the proofs up to Section \ref{sec:app} have been carried out in the proof assistant Minlog. Furthermore automatic generation of correctness proofs and translation to Haskell was carried out. Even though the proofs from \ref{sec:app} have only been partially carried out in a proof assistant, the program extraction by hand was still a reliable method to get certified algorithms.
	
	Although algorithms extracted via the indirect method by translations tend to have a low lookahead, they do rely on rational arithmetic, so the direct method should outperform the indirect one in most cases. Our aim was to obtain verified algorithms and we do not claim that our programs are the most efficient. Some inefficiency stems from overestimation of bounds in formal proofs. These can usually be removed by careful analysis of the proofs involved.
	
	Since we have proven that the signed digit representation is closed under multiplication, average, division and limits we can now easily prove that a lot of functions are represented by stream-transformers, e.g.~trigonometric functions, using their Taylor-expansion directly as was done for Herons algorithm.
	Another viable approach to limits should be to use the completeness of metric spaces, i.e.~prove that signed-digits-reals satisfy the axioms of a metric space and then use a completeness theorem for abstract metric spaces.   
\section*{Acknowledgment}
The first author would like to thank  the Istituto Nazionale di Alta Matematica ``Francesco Severi" for their scholarship of his PhD study. 
This project has received funding from the European Unions Horizon 2020 research
and innovation programme under the Marie Sk{\l}odowska-Curie grant agreement No.~731143. In addition it was funded by the FWF project P 32080-N31.\\
Both authors would like to thank Helmut Schwichtenberg for proof reading this paper and for his support during the creation of this paper.
\bibliographystyle{alphaurl}
\bibliography{Sources}
\end{document}